\documentclass[11pt]{article}
\usepackage{graphicx}
\usepackage{amssymb}
\usepackage{bm}
\usepackage{mathtools}
\usepackage{amsmath}
\usepackage{commath}
\usepackage{amsthm}
\usepackage{enumitem}
\usepackage{float}
\usepackage{array}
\usepackage{outlines}
\usepackage{setspace}
\usepackage[colorinlistoftodos]{todonotes}
\setlist[enumerate,1]{label=\alph*)}
\setlist[enumerate,2]{label=\roman*)}
\newtheorem{theorem}{Theorem}

\newtheorem{remark}{Remark}
\newtheorem{definition}{Definition}
\newtheorem{prop}{Proposition}
\newcommand*\diff{\mathop{}\!\mathrm{d}}
\usepackage{tikz}
\usetikzlibrary{arrows,shapes,trees, calc}

\title{Weak solution of the merged mathematical equations of the polluted atmosphere}

\author{
\large{\bf{Donatella Donatelli} and \bf{N\'ora Juh\'asz}
\thanks{The research of DD and NJ leading to these results has received funding from the European Union's Horizon 2020 Research and Innovation Programme under the Marie Sklodowska-Curie Grant Agreement No 642768 (Project Name: ModCompShock).}} \\[1ex]
\normalsize Department of Information Engineering, Computer Science and Mathematics \\
\normalsize University of L'Aquila \\
\normalsize 67100 L'Aquila, Italy. \\
\normalsize {donatella.donatelli@univaq.it}, {norjuh@univaq.it}
}

\date{}

\begin{document}

\maketitle

\begin{abstract}

Considered as a geophysical fluid, the polluted atmosphere shares the shallow domain characteristics with other natural large-scale fluids such as seas and oceans. This means that its domain is excessively greater horizontally than in the vertical dimension, leading to the classic hydrostatic approximation of the Navier-Stokes equations. The authors of the \cite{azerad2001mathematical} article have proved a convergence theorem for this model with respect to the ocean, without considering pollution effects. The novelty of this present work is to provide a generalisation of their result translated to the atmosphere, extending the fluid velocity equations with an additional convection-diffusion equation representing pollutants in the atmosphere. 

\end{abstract}

{\bf Keywords:} Navier-Stokes equations, shallow domains, pollution evolution equation, hydrostatic approximation, compactness, weak solutions.

\newpage

\tableofcontents
\newpage

\section{Introduction} \label{Introduction}

The key aspects of capturing the dynamics of either water flow in oceanography or atmospheric changes in meteorology are the following two fundamental concepts that underlie many modern asymptotic models aiming to describe them. The first one is that both phenomena can be viewed as a fluid dynamics process, and as such, they are well described by the Navier-Stokes equations. The second is the notion of the so-called shallow domains. The latter is a widely used concept in the field of large-scale geophysical fluids and it takes advantage of their approximately flat structure when modelling them with the Navier-Stokes equations. Essentially, the basic principle is to exploit the fact that when we are working on domains such as the ocean or the atmosphere, we are working on an "almost" two-dimensional set, and thus we can significantly simplify the model in the vertical direction, or even eliminate the third dimension and arrive to a limit model in two dimensions.
A common practice for example is to depth-integrate the Navier-Stokes equations for big lakes and seas and obtain the shallow water equations.
Our approach is different, we keep the vertical velocity and will arrive to a three dimensional limit model.
We use the
\[ \epsilon = \frac{\text{characteristic depth}}{\text{characteristic width}} \]
aspect ratio and the fact that, as the above described general ideas suggest, it is natural to consider asymptotic models and to take the limit model as $\epsilon$ goes to zero.

We consider the phenomena of a contaminant being emitted into, mixing with, and being deposited from the clean air. This topic is in the intersection of widely used fields such as mathematical fluid dynamics, atmosphere modelling, and meteorology; for a profound insight see e.g. the books \cite{arya2001introduction}, \cite{pedlosky2013geophysical} and \cite{temam2001navier}.

Our focus is to create a self-contained model describing the atmospheric flow combined with the pollution effects and to investigate the asymptotic behaviour of the weak solutions of this model. The model itself, new to our knowledge, utilises classical concepts as its foundation and it merges them with relatively new or more specific ideas. On the one hand, we use general and well-known approaches such as the wind flow being modelled by the incompressible anisotropic Navier-Stokes equations, and using a convection-diffusion equation to describe the pollution effect. On the other hand, the model also incorporates notions like using a Gaussian pulse as a source term (\cite{zhuk2016source}), and rescaling the diffusivity parameters according to the largely different horizontal and vertical scales (similarly to the way the viscosity terms are scaled in \cite{azerad2001mathematical}).

Specifically, as for the limit behavior of the weak solutions of the model we construct, we show that the weak solutions of this model converge to a weak solution of the hydrostatic limit model --- in other words, taking the limit we arrive to a justification of the hydrostatic model for the case of the polluted atmosphere, where, instead of the originally used wind momentum equation, the hydrostatic air pressure assumption is used to describe the process in the vertical dimension. Since the polluted atmosphere as a merged physical phenomenon can naturally be seen as the air-analogous version of the salty sea, with this convergence result we are providing an extension of the main theorem of \cite{azerad2001mathematical} in the sense that we obtain a result describing the atmospheric flow combined with pollution effects, while the authors of \cite{azerad2001mathematical} provide a similar result for the ocean excluding salinity.

This paper is organised in the following way. In Section \ref{eqsdescribingatmpluspollution} we describe the physical model of the polluted atmosphere and the scaling leading to the hydrostatic equations. In Section \ref{weakformulationmainresult} we introduce the function spaces we work in, and we present the weak formulation and the main theorem. The main theorem's proof is described in Section \ref{proof}. A few concluding remarks can be found in Section \ref{concludingremarks}, finally in the appendix of the article we present an existence result for the hydrostatic equations of the polluted atmosphere.

We end up this section with some notations we will use through the paper.

\subsection{Notations}

\begin{enumerate}

\item $\epsilon = $ the aspect ratio,
\item $\nu = $ viscosity,
\item $\Omega_\epsilon, \Omega = $ the original and the rescaled ($\epsilon$-independent) domain, respectively,
\item $\nabla$ stands for the gradient vector; specifically, $(\partial_x, \partial_y, \partial_z)$ before, and $(\partial_1, \partial_2, \partial_3)$ after rescaling,
\item $\nabla_\nu = (\nu_x^{1/2} \partial_x, \nu_y^{1/2} \partial_y, \nu_z^{1/2} \partial_z)$ on $\Omega_\epsilon$, $\nabla_\nu = (\nu_1^{1/2} \partial_1, \nu_2^{1/2} \partial_2, \nu_3^{1/2} \partial_3)$ on the rescaled domain $\Omega,$
\item $ \Delta_\nu$ stands for the anisotropic Laplacian operators $\nu_x \partial^2_{xx} + \nu_y \partial^2_{yy} + \nu_z \partial^2_{zz}$ and $\nu_1 \partial^2_{11} + \nu_2 \partial^2_{22} + \nu_3 \partial^2_{33}$ on the original and rescaled domains, respectively,
\item $\mathbf{g}$ represents the force due to gravity, which also includes the centrifugal effect,
\item $\varphi = $ the gravity potential term, i.e. $\mathbf{g} = \nabla \varphi,$
\item $\mathbf{w}$ represents the Earth rotation angular speed, $l(y) = $ latitude,
\item $\alpha = 2 f \sin (l(x_2))$, $\beta = 2 f \cos (l(x_2)),$
\item $\mathbf{K}$ represents the diffusion matrix,
\item $( \cdot,\cdot) = $ the scalar product in $L^2(\Omega)^d$ or the duality $L^p (\Omega)$, $L^{p'} (\Omega),$
\item $\mathbf{u}_H$ = the horizontal velocity components $(u_1, u_2),$
\item $b(\mathbf{u}_H)$ = $\alpha (-u_2 , u_1),$
\item $L^p_t L^q_x$ is the abbreviated notation for $L^p (0, T; L^q(\Omega)).$

\end{enumerate}

\section{The polluted atmosphere model} \label{eqsdescribingatmpluspollution}

We begin by choosing a coordinate system, set the domain we start with and an appropriate set of equations describing both atmosphere motions and the effect of pollution. We ignore large scale effects such as the curvature of the Earth. For describing the effect of pollution we use a continuity equation written in Cartesian coordinates for the concentration.

We fix a local Cartesian coordinate system where the axes are independent from the wind ($x, y$ and $z$ are oriented towards east, north, and upwards, respectively). Other possibilities and considerations regarding choosing the coordinate system are described briefly in the concluding remarks.

The domain we consider is a local slice of the atmosphere filled with air that we assume to be incompressible, and a pollutant of concentration $C,$ therefore it is defined by the set
\begin{equation}
\Omega_{\epsilon} = \{(x,y,z) \in \mathbb{R}^3 ; (x,y) \in \omega, 0 < z < \epsilon h(x,y) \},
\end{equation}
where $\omega$ is a Lipschitz-domain in $\mathbb{R}^2$, $h$ is a nonnegative Lipschitz-continuous function, and $\epsilon$ incorporates the depth of the domain.

The equations governing the air velocity $\mathbf{v}$ and pressure $q$ in the atmosphere are the incompressible Navier-Stokes equations, in which we use different viscosities according to vertical or horizontal directions, and the density is taken to be identically equal to one:
\begin{equation} \label{originalEQ-V}
\partial_t \mathbf{v} + (\mathbf{v} \cdot \nabla)\mathbf{v} -\Delta_{\nu} \mathbf{v} + \nabla q + 2 \mathbf{w} \times \mathbf{v}= \mathbf{g} \text{ in } \Omega_{\epsilon} \times (0,T)
\end{equation}
\begin{equation} \label{originalEQ-I}
\nabla \cdot \mathbf{v} = 0 \text{ in } \Omega_{\epsilon} \times (0,T)
\end{equation}
Contaminants emitted into the atmosphere are transported by the wind, mixed and dispersed by turbulent behaviour, and deposited onto the ground by gravity. The dispersion of pollutants is commonly modelled by a diffusive equation of the form
\begin{equation} \label{diffusive}
\partial_t P + \mathbf{v} \cdot \nabla P = \nabla \cdot (\mathbf{K} \nabla P) + Q,
\end{equation}
where $P$ is the pollution concentration and $\mathbf{K}$ is the $3 \times 3$ diffusion matrix --- the diagonal terms represent the strength of the diffusion in the $x, y, z$ directions, while the non-diagonal terms reflect the correlation of random motions between each pair of principal directions.

Here $Q$ represents the sources of the pollutant in the atmosphere i.e. its emission, removal from the atmosphere by dry and wet deposition, and chemical reactions.

Now, we perform a vertical scaling to make the domain independent of $\epsilon.$

\[x=x_1, \quad y=x_2, \quad z=\epsilon x_3. \]

For simplicity of notation we will use $\partial_i = \partial_{x_i} $ for $i = 1,2,3.$

We have the following new, $\epsilon$ - independent domain

\[ \Omega = \{(x_1,x_2,x_3) \in \mathbb{R}^3 ; (x_1,x_2) \in \omega, 0 < x_3 < h (x_1 , x_2) \}.\]

The corresponding scaling regarding the kinematic, pressure and viscosity terms are

\[v_x = u_1^{\epsilon}, \quad v_y = u_2^{\epsilon}, \quad v_z = \epsilon u_3^{\epsilon}, \quad p = p^{\epsilon}, \]
\[\nu_x = \nu_1, \quad \nu_y = \nu_2 , \quad \nu_z = \epsilon ^ 2\nu_3, \]
where the latter embodies the different eddy viscosity hypothesis that we mentioned earlier, and $p$ is the pressure term including the gravity potential $\varphi,$ i.e., $p = q - \varphi$. Basically the scaling of the viscosity parameters can be heuristically deduced from the viscosity dimensions: $[\nu_{x_i}] = \text{typical length}_{x_i}^2 / \text{typical time}$, more details on the subject can be found in \cite{besson1992some}, \cite{azerad2001mathematical}.

Note that from this point we have $\nabla = (\partial_1, \partial_2, \partial_3) $ and $ \Delta_\nu = \nu_1 \partial^2_{11} + \nu_2 \partial^2_{22} + \nu_3 \partial^2_{33} $.

We use an analogous rescaling methodology for the diffusivity constants in (\ref{diffusive}),

\begin{gather}
\mathbf{K}
=
\begin{bmatrix} K_{xx} & K_{xy} & K_{xz} \\ K_{yx} & K_{yy} & K_{yz} \\ K_{zx} & K_{zy} & K_{zz} \end{bmatrix}
=
\begin{bmatrix}
M_{11} & M_{12} & \epsilon M_{13} \\ M_{21} & M_{22} & \epsilon M_{23} \\ \epsilon M_{31} & \epsilon M_{32} & \epsilon^2 M_{33}
\end{bmatrix},
\end{gather}
where

\begin{gather}
\mathbf{M}
=
\begin{bmatrix}
M_{11} & M_{12} & M_{13} \\ M_{21} & M_{22} & M_{23} \\ M_{31} & M_{32} & M_{33}
\end{bmatrix}.
\end{gather}

Here we emphasise that the diffusivity matrix $\mathbf{M}$ is assumed to have the coercivity property
\begin{equation} \label{coercivity}
\lambda \norm{\mathbf{x}}^2 \leq (\mathbf{M} \mathbf{x}, \mathbf{x})
\end{equation}
for all $\mathbf{x} \in \mathbb{R}^3$ with the coercivity constant $\lambda > 0.$ In other words we require that $\lambda = \min {\lambda_i} > 0$ where $\lambda_i$ are the eigenvalues of $\mathbf{M}.$

Note that we do not assume the air to be homogeneous, but for completeness we highlight that the specific case of homogeneous air corresponds to $\mathbf{K}$ taking the form

\begin{gather} \nonumber
\mathbf{K}
=
\begin{bmatrix}
1 & 0 & 0 \\ 0 & 1 & 0 \\ 0 & 0 & \epsilon^2
\end{bmatrix}.
\end{gather}

Finally we obtain the rescaling equations for the pollution concentration and the source term by considering their physical dimensions, namely: $\text{typical mass} / \text{typical volume},$ from which we deduce
\begin{equation}
P = \frac{1}{\epsilon} C^\epsilon, \quad Q = \frac{1}{\epsilon} S^\epsilon.
\end{equation}
In this model we will use localised point sources as suggested in \cite{zhuk2016source} and we use what is called a parametrised approximated delta function $\delta_\kappa$ (a Gaussian distribution, specifically) to describe the source term $S^\epsilon$.

Let us consider a contaminant plume with intensity $I$ originating at emission coordinates $\mathbf{x}_s$ beginning at time $t_s$. At the beginning we assume zero concentration.
An abrupt emission is modelled by a source term represented as a product of an approximated delta function $\delta_\kappa (\mathbf{x} - \mathbf{x}_s)$ (centered at $\mathbf{x}_s$ with a constant parameter $\kappa$) and a scaled Heaviside function $s$ centred at time (switching from 0 to $I$ at time $t_s.$)
In fact, this source term models the continuous emission of the contaminant at constant rate and fixed location.
Basically $s(t;t_s)$ is a function which activates the source, namely $s(t;t_s) := IH(t - t_s).$ We arrive to the complete source term model of the form
\begin{equation} \label{approxsourceterm}
S_\kappa (t, t_s, x, x_s) = s(t; t_s) \delta_\kappa (\mathbf{x} - \mathbf{x}_s),
\end{equation}
while in the limit model the source term $S$ naturally stands for the Dirac $\delta$ distribution.

Eventually, as we mentioned before, we use a Gaussian to approximate the delta function to represent the source, so it is convenient to reformulate (\ref{approxsourceterm}) in terms of $\epsilon.$ In particular, we choose the latter to be in a fixed proportion to the standard deviation parameter $\sigma$ of the Gaussian distribution describing the source (specifically, $2\sigma^2 = \epsilon^2$). In other words we set $\kappa$ to be $\epsilon$ and the following function will give shape to the source term:

\[\delta_\epsilon (\mathbf{x}-\mathbf{x_s}) = \gamma \epsilon^{-3} e ^ {\frac{-\abs{\mathbf{x}-\mathbf{x_s}}^2}{\epsilon^2}};\]
and finally
\begin{equation} \label{sourcedefinition}
S^\epsilon = I H(t-t_s) \gamma \epsilon^{-3} e ^ {\frac{-\abs{\mathbf{x}-\mathbf{x_s}}^2}{\epsilon^2}},
\end{equation}
where $\gamma$ is a normalising constant. Later we will take the limit as $\epsilon$ goes to zero, and the above setting will result in having the $\delta$ distribution representing the hydrostatic source term, which is a physically natural choice.

Putting this condition and the previously described rescaling terms together to the original equations (\ref{originalEQ-V})-(\ref{diffusive}), we arrive to the merged anisotropic equations of the polluted atmosphere:
\begin{equation} \label{ANSC1}
\partial_t u^\epsilon _1 + \mathbf{u}^\epsilon \cdot \nabla u^\epsilon _1 -\Delta_\nu u^\epsilon _1 -\alpha u_2^\epsilon + \epsilon \beta u_3^\epsilon + \partial_1 p^\epsilon = 0 \text{ in } \Omega \times (0,T)
\end{equation}
\begin{equation} \label{ANSC2}
\partial_t u_2^\epsilon + \mathbf{u}^\epsilon \cdot \nabla u_2^\epsilon -\Delta_\nu u_2^\epsilon + \alpha u_1^\epsilon + \partial_2 p^\epsilon = 0 \text{ in } \Omega \times (0,T)
\end{equation}
\begin{equation} \label{ANSC3}
\epsilon^2 \{ \partial_t u_3^\epsilon + \mathbf{u}^\epsilon \cdot \nabla u_3^\epsilon -\Delta_\nu u_3^{\epsilon} \} - \epsilon \beta u_1^\epsilon + \partial_3 p^\epsilon = 0 \text{ in } \Omega \times (0,T)
\end{equation}
\begin{equation} \label{ANSC4}
\nabla \cdot \mathbf{u}^\epsilon= 0 \text{ in } \Omega \times (0,T)
\end{equation}
\begin{equation} \label{ANSC5}
\partial_t C^\epsilon + \mathbf{u^\epsilon} \cdot \nabla C^\epsilon = \nabla \cdot (\mathbf{M} \nabla C^\epsilon) + S^\epsilon \text{ in } \Omega \times (0,T)
\end{equation} 

The initial condition for the velocity and the pollution concentration are
\begin{equation} \label{ANSC_init}
\mathbf{u}^\epsilon (\cdot, t = 0) = \mathbf{u}_0, \quad C^\epsilon (\cdot, t=0) = C_0 \quad \text{ in } \Omega.
\end{equation}

To make the above set of equations complete, we need to assign the boundary conditions.

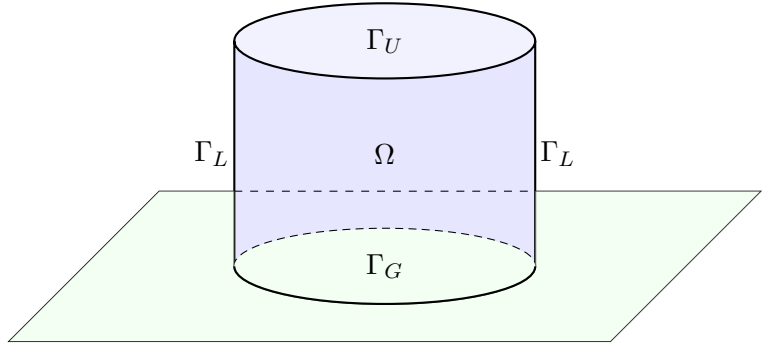
\begin{figure}[!h]
\centering
\begin{tikzpicture}
\fill[color=blue!10] (2,0) -- (2,3) arc (360:180:2cm and 0.5cm) -- (-2,0) arc (180:360:2cm and 0.5cm);
\fill[color=blue!5] (0,3) circle (2cm and 0.5cm);
\draw[thick] (-2,3) -- (-2,0) arc (180:360:2cm and 0.5cm) -- (2,3) ++ (-2,0) circle (2cm and 0.5cm);
\draw[densely dashed, thick] (-2,0) arc (180:0:2cm and 0.5cm);
\draw (-2,1) -- (-3,1) --(-5,-1) -- (3,-1) -- (5,1) -- (2,1);
\draw[dashed] (-2,1) -- (2,1);
\fill[green!10,opacity=0.5] (-2,0) -- (-2,1) -- (-3,1) -- (-5,-1) -- (3,-1) -- (5,1) -- (2,1) -- (2,0) arc (0:180:2cm and -0.5cm);
\fill[color=green!5] (0,0) circle (2cm and 0.5cm);
\draw (0,1.5) node {$\Omega$};
\draw (0,3) node {$\Gamma_U$};
\draw (0,0) node {$\Gamma_G$};
\draw (-2.3,1.5) node {$\Gamma_L$};
\draw (2.3,1.5) node {$\Gamma_L$};
\draw[thick] (2,0) arc (360:180:2cm and 0.5cm);
\end{tikzpicture}
\caption{Boundary structure.}
\label{Fig1}
\end{figure}

Concerning the pollution concentration it is vital to choose boundary conditions that are as physically relevant as possible and yet simple enough. As we can see in many related papers, in the case of LAMs (limited area models) finding the right boundary conditions can be challenging.

In order to describe the boundary conditions precisely we divide $\Gamma = \partial \Omega$ as follows. The upper boundary of the $\Omega$ domain is denoted by $\Gamma_U$, the lateral boundary is $\Gamma_L$, the lower boundary of the domain is $\Gamma_G,$ and for the collective $\Gamma_L \cup \Gamma_U$ section we use the notion $\Gamma_A$ (above the ground), see Figure \ref{Fig1}.

We highlight that in the case of geophysical coordinates we are considering a generic landscape below the air. This means that we equally allow inland and coastal boundary conditions on the lower boundary $\Gamma_G,$ but for simplicity we do require that pollutants do not sink below ground level.

The boundary conditions we choose to use in this case are the following:

\begin{figure}[!h]
  \centering
  \begin{tikzpicture}
  
    \draw[blue] (0,0) -- (2,0) -- (4,0) -- (4,2) -- (4,4) -- (2,4) -- (0,4) -- (0,2) -- (0,0);
    \draw[green, thick] (0,0) -- (4,0);
    
    \draw (2,2)node{$\Omega$};
    \draw (2,0)node[below] {$\nu_3 \partial_3 \mathbf{u}^\epsilon_H = \theta_H,$ $ u_3^\epsilon = 0$ and $\mathbf{M} \nabla C^\epsilon \cdot \vec{\boldsymbol{n}}_{\Gamma_G} = 0$};
    \draw (4,2)node[right] {$\mathbf{u^\epsilon} = 0, C^\epsilon=0 $};
    \draw (2,4)node[above] {$\mathbf{u^\epsilon} = 0, C^\epsilon=0 $};
    \draw (0,2)node[left] {$\mathbf{u^\epsilon} = 0, C^\epsilon=0 $};
    
  \end{tikzpicture}
  \caption{Boundary conditions.}
\label{BC-Fig}
\end{figure}
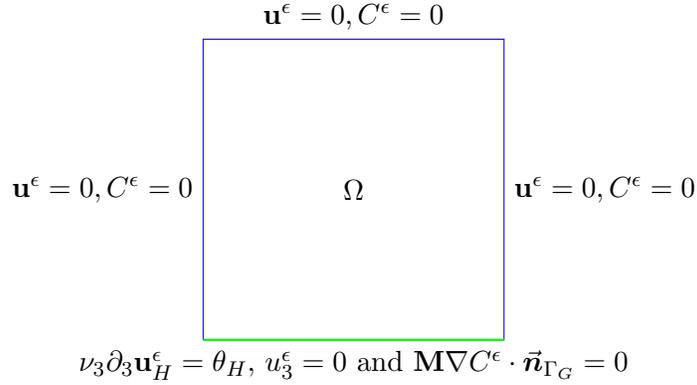

\begin{itemize}
  
  \item The upper boundary $\Gamma_U$ of the domain we are considering is chosen to be at the isobar $z=h$ representing the planetary boundary layer (see \cite{arya2001introduction}), therefore we assume
    
    \begin{equation}
      \mathbf{u}^\epsilon_H = 0, \quad u^\epsilon_3 = 0, \quad C^\epsilon = 0 \quad \text{ on } \Gamma_U \times (0,T).
    \end{equation}
      
  \item On $\Gamma_L$ we use the same boundary conditions as on the upper level, i.e. we have
    
    \begin{equation}
      \mathbf{u}^\epsilon_H = 0, \quad u^\epsilon_3 = 0, \quad C^\epsilon = 0  \quad \text{ on } \Gamma_L \times (0,T).
    \end{equation}
    
    \begin{remark}
    If we wanted to use $C^\epsilon(x,y,z) = \kappa$ with some $\kappa$ constant value, we could make a change of variables subtracting that background concentration and arrive back to the same mathematical problem with $C^\epsilon$ being zero.
    \end{remark}
    
  \item In this model we neglect topological variations or large waves, and the $z=0$ plane represents the lower boundary. On the ground level the boundary conditions are defined as follows. For the vertical air velocity we have the
  
    \begin{equation}
      u^\epsilon_3 = 0 \quad \text{ on } \Gamma_G \times (0,T)
    \end{equation}
condition, which represents the impervious nature of the ground with respect to the wind. For the horizontal velocities we set

    \begin{equation}
      \nu_3 \partial_3 \mathbf{u}^\epsilon_H = \theta_H \quad \text{ on } \Gamma_G \times (0,T),
    \end{equation}
where, roughly speaking, the $\theta_H = (\theta_1, \theta_2)$ parameter represents how "smooth" the braking effect of the terrain is. For example if we consider air flow above a water surface, in this case there is a $\theta_H$ degree of freedom in the slowing process of the wind: as the horizontal air layers are closer and closer to the $\Gamma_G$ lower boundary, $\mathbf{u}_H^\epsilon$ gradually becomes smaller, but it does not necessarily disappear. In other words, above water bodies the braking effect is not immediate: horizontal wind traction generates movement in the water, which leaves space for air flow even on the boundary. On the other hand, above a rocky inland surface this degree of freedom is zero, as the no-slip boundary conditions force the air flow to a full stop on the boundary. Note that it is possible to introduce this parameter at an earlier point in the model development process: if we define the braking function as an $\upsilon_H$ parameter before the scaling, we need to apply the $\upsilon_H = \epsilon \theta_H$ scaling equation (for details see \cite{azerad2001mathematical}). As for the concentration, we assume that the pollution is not absorbed by the ground, so we use the
    
    $$\mathbf{M} \nabla C^\epsilon \cdot \vec{\boldsymbol{n}}_{\Gamma_G} = 0 \quad \text{ on } \Gamma_G$$
    
    boundary condition (see also for example \cite{yeh1975three}), where $\vec{\boldsymbol{n}}_{\Gamma_G}$ is the outward normal vector to $\Gamma_G$. More explicitly, this means $M_{31}\partial_1 C^\epsilon + M_{32} \partial_2 C^\epsilon + M_{33} \partial_3 C^\epsilon = 0$ on $\Gamma_G.$
  
\end{itemize}

Finally, the boundary conditions to describe the anisotropic system (\ref{ANSC1}) - (\ref{ANSC_init}) can be summarised as follows:
\begin{equation} \label{ANSC_bc}
\begin{split}
& \mathbf{u}^\epsilon_H = 0, u^\epsilon_3 = 0, C^\epsilon = 0 \text{ on } \Gamma_A \times (0,T), \\
& \nu_3 \partial_3 \mathbf{u}^\epsilon_H = \theta_H, u^\epsilon_3 = 0 \text{ on } \Gamma_G \times (0,T), \\
& \mathbf{M} \nabla C^\epsilon \cdot \vec{\boldsymbol{n}}_{\Gamma_G} = 0 \text{ on } \Gamma_G \times (0,T).
\end{split} \raisetag{40pt}
\end{equation}
Note that we have a domain where the lower and upper boundaries represent a physically existent layer (i.e. the ground and the planetary boundary layer), but outside the lateral boundary the physical world continues without influence and without any particular external condition. On boundaries of such nature we have to apply OBCs (open boundary conditions), which is not a trivial task, see for example \cite{wang2014ocean}. We need to avoid any spurious reflection or constraint which is unnatural, but at the same time we want the model to remain mathematically manageable.

If we assume $\mathbf{u}^\epsilon = O(1)$, then neglecting the $\epsilon^2$ and $\epsilon$ in the anisotropic equations (\ref{ANSC1}) - (\ref{ANSC_init}), (\ref{ANSC_bc}), we formally arrive to the following hydrostatic Navier-Stokes equations (i.e., primitive equations) combined with pollution:
\begin{equation} \label{HNSC1}
\partial_t u_1 + \mathbf{u} \cdot \nabla u_1 - \Delta_\nu u_1 -\alpha u_2 + \partial_1 p = 0 \text{ in } \Omega \times (0,T)
\end{equation}
\begin{equation} \label{HNSC2}
\partial_t u_2 + \mathbf{u} \cdot \nabla u_2 - \Delta_\nu u_2 +\alpha u_1 + \partial_2 p = 0 \text{ in } \Omega \times (0,T)
\end{equation}
\begin{equation} \label{HNSC3}
\partial_3 p = 0 \text{ in } \Omega \times (0,T)
\end{equation}
\begin{equation} \label{HNSC4}
\nabla \cdot \mathbf{u} = 0 \text{ in } \Omega \times (0,T)
\end{equation}
\begin{equation} \label{HNSC5}
\partial_t C + \mathbf{u} \cdot \nabla C = \nabla \cdot (\mathbf{M} \nabla C) + S \text{ in } \Omega \times (0,T)
\end{equation}
\begin{equation} \label{HNSC_init}
u_i (\cdot, t=0) = u_{0 i}, C (\cdot, t=0) = C_0 \text{ in } \Omega, \text{$i$ = 1,2}
\end{equation}
\begin{equation} \label{HNSC_bc}
\begin{split}
& \mathbf{u}_H = 0, u_3 n_3 = 0, C = 0 \text{ on } \Gamma_A \times (0,T), \\
& \nu_3 \partial_3 \mathbf{u}_H = \theta_H, u_3 = 0 \text{ on } \Gamma_G \times (0,T), \\
& \mathbf{M} \nabla C \cdot \vec{\boldsymbol{n}}_{\Gamma_G} = 0 \text{ on } \Gamma_G \times (0,T).
\end{split}
\end{equation}
\begin{remark} \label{sourceregularityremark}
Here $S$ stands for the $\delta$ limit distribution. We keep the notation $S$ for generality, since as we will see later, instead of the approximated delta functions and the delta distribution, it is possible to use any sort of "sufficiently smooth" source term (which can be independent of $\epsilon$ in the anisotropic case as well) that is bounded in $L^2_t L^2_x.$
\end{remark}

\section{Weak formulation and Main Result} \label{weakformulationmainresult}

In this section we describe the weak formulation of the merged equations in both the anisotropic and hydrostatic case we are going to use in the paper, and we state our main result.

We need to define the following spaces:
\[ C^{\infty}_{\Gamma_A} (\Omega) = \{ \varphi \in C^\infty (\bar{\Omega}); \varphi = 0 \text{ on some neighbourhood of } \Gamma_A \} \]
\[ H^k_{\Gamma_A} (\Omega) = \overline{C_{\Gamma_A} ^ \infty (\Omega)} ^ {H^k (\Omega)} = \{ v \in H^k(\Omega); \partial^\alpha v=0 \text{ on } \Gamma_A \text{ for any } \abs{\alpha} < k \} \]
\[ \mathbf{V} = \{ \mathbf{v} \in H^1_{\Gamma_A} (\Omega) \times H^1_{\Gamma_A} (\Omega) \times H^1_0 (\Omega); \nabla \cdot \mathbf{v} = 0 \text{ in } \Omega \} \]
\[ H (\partial_3, \Omega) = \{ v \in L^2(\Omega); \partial_3 v \in L^2 (\Omega)\} \]
\[ H_0 (\partial_3, \Omega) = \overline{C_0 ^\infty (\Omega)}^{H (\partial_3, \Omega)} = \{ v \in H(\partial_3, \Omega); v n_3 = 0 \text{ on } \Gamma \} \]
\[ \mathbf{W} = \{ \mathbf{v} \in H^1_{\Gamma_A} (\Omega) \times H^1_{\Gamma_A} (\Omega) \times H_0 (\partial_3, \Omega); \nabla \cdot \mathbf{u} = 0 \text{ in } \Omega \} \]

Furthermore we introduce the notation $b(\mathbf{u}_H)$ = $\alpha (-u_2 , u_1).$

Then the weak formulation of the hydrostatic system (\ref{HNSC1}) - (\ref{HNSC_bc}) takes the following form.

\begin{definition} \label{definitionhydrostaticWF}

The pair $(\mathbf{u}, C)$ is called a weak solution of the a hydrostatic system (\ref{HNSC1}) - (\ref{HNSC5}) subject to (\ref{HNSC_init}) - (\ref{HNSC_bc}) if $\mathbf{u} = (\mathbf{u}_H, u_3) \in L^2 (0, T; \mathbf{W})$ with $\mathbf{u}_H \in L^\infty (0,T; L^2(\Omega)^2),$ and $C \in L^\infty(0,T; L^2(\Omega)) \cap L^2(0,T; H^1_{\Gamma_A}(\Omega))$, moreover $\mathbf{u}$ and $C$ satisfy the integral identities
\begin{equation}
\begin{split}
& \int_0^T \bigg[ -(\mathbf{u}_H, \partial_t \mathbf{\tilde{u}}_H) + (\nabla_\nu \mathbf{u}_H, \nabla_\nu \mathbf{\tilde{u}}_H) - (\mathbf{u}_H, (\mathbf{u} \cdot \nabla) \mathbf{\tilde{u}}_H) + (b(\mathbf{u}_H), \mathbf{\tilde{u}}_H) \bigg ] \diff t \\
& = ({\mathbf{u}_H}_0, \mathbf{\tilde{u}}_H(0)) - \int_0^T \langle \theta_H, \mathbf{\tilde{u}}_H \rangle_{\Gamma_G} \diff t \raisetag{15pt}\label{hydrostaticWF_U}
\end{split}
\end{equation}

and

\begin{equation}
\begin{split}
& \int_0^T \bigg[ -(C, \partial_t \tilde{C}) -(\mathbf{u}C, \nabla \tilde{C}) +  (\mathbf{M} \nabla C, \nabla \tilde{C}) \bigg ] \diff t \\
& = (C_0, \tilde{C}(0)) + \int_0^T (S,\tilde{C}) \diff t \raisetag{20pt} \label{hydrostaticWF_C}
\end{split}
\end{equation}
for all $(\mathbf{\tilde{u}}, \tilde{C})$ with $\mathbf{\tilde{u}} = (\mathbf{\tilde{u}}_H, \tilde{u}_3) \in H^1 (0,T, \mathbf{W})$, $\mathbf{\tilde{u}}_H (T) = 0$ and $\partial_3 \mathbf{\tilde{u}}_H \in L^\infty (0, T; L^3 (\Omega)^2)$ and $\tilde{C} \in L^2 (0,T, H^2_{\Gamma_A} (\Omega)) \cap H^1 (0,T, L^2)$, $\tilde{C} (T) = 0$. 
\end{definition}

Note that the term $\int_0^T (S,\tilde{C}) \diff t$ makes sense since in this case $S$ denotes the delta distribution and we have 
\[ (S,\tilde{C}) = \int_\Omega \delta(\mathbf{x} - \mathbf{x}_s) \tilde{C}(\mathbf{x}) \diff \mathbf{x} = \tilde{C}(\mathbf{x}_s).\]
The existence of weak solution for the hydrostatic system (\ref{HNSC1}) - (\ref{HNSC_bc}) can be proved in the spirit of \cite{temam2005some}. In appendix A we will outline the main steps of the proof and the differences compared to the ocean model.

The weak form of the anisotropic system (\ref{ANSC1}) - (\ref{ANSC_init}), (\ref{ANSC_bc}) is as follows.

\begin{definition} \label{definitionanisotropicWF}

The pair $(\mathbf{u}^\epsilon, C^\epsilon)$ is called a weak solution of the anisotropic system (\ref{ANSC1}) - (\ref{ANSC5}) subject to (\ref{ANSC_init}), (\ref{ANSC_bc}) if $\mathbf{u}^\epsilon = (\mathbf{u}^\epsilon_H, u^\epsilon_3) \in L^2 (0,T ; \mathbf{V}) \cap L^\infty (0, T; L^2 (\Omega)^3)$ and $C^\epsilon \in L^\infty (0,T; L^2(\Omega)) \cap L^2 (0,T; H^1_{\Gamma_A}(\Omega))$, moreover $\mathbf{u}^\epsilon$ and $C^\epsilon$ satisfy the integral identities
\begin{equation}
\begin{split}
& \int_0^T \bigg [ -(\mathbf{u}^\epsilon_H, \partial_t \mathbf{\tilde{u}}_H) + (\nabla_\nu \mathbf{u}^\epsilon_H, \nabla_\nu \mathbf{\tilde{u}}_H) - (\mathbf{u}^\epsilon_H, (\mathbf{u}^\epsilon \cdot \nabla) \mathbf{\tilde{u}}_H) + (b(\mathbf{u}^\epsilon_H), \mathbf{\tilde{u}}_H) \bigg ] \diff t \\
& + \epsilon \int_0 ^ T \big [ (\beta u_3 ^\epsilon, \tilde{u}_1) - (\beta u_1^\epsilon, \tilde{u}_3) \big ] \diff t\\
& + \epsilon^2 \int _0 ^ T \big [ -(u_3^\epsilon, \partial_t \tilde{u}_3) +(\mathbf{u}^\epsilon \cdot \nabla u_3 ^ \epsilon, \tilde{u}_3) + (\nabla_\nu u_3^\epsilon, \nabla_\nu \tilde{u}_3) \big ] \diff t \\
& = ({\mathbf{u}_H}_0, \mathbf{\tilde{u}}_H(0)) + \epsilon^2 (u_{03}^\epsilon, \tilde{u}_3 (0)) - \int_0^T \langle \theta_H, \mathbf{\tilde{u}}_H \rangle_{\Gamma_G} \diff t \raisetag{20pt} \label{anisotropicWF_U}
\end{split}
\end{equation}
and
\begin{equation}
\begin{split}
& \int_0^T \bigg [ -(C^\epsilon, \partial_t \tilde{C}) -(\mathbf{u}^\epsilon C^\epsilon, \nabla \tilde{C}) + (\mathbf{M} \nabla C^\epsilon, \nabla \tilde{C}) \bigg ] \diff t \\
& =  (C_0, \tilde{C} (0)) + \int_0^T (S^\epsilon,\tilde{C}) \diff t \raisetag{20pt} \label{anisotropicWF_C}
\end{split}
\end{equation}
for all $\mathbf{\tilde{u}} = (\mathbf{\tilde{u}}_H, \tilde{u}_3) \in H^1 (0,T;\mathbf{V})$ with $\mathbf{\tilde{u}}(T) = 0$ and $\tilde{C}$ such that $\tilde{C} \in H^1 (0,T, H^2_{\Gamma_A} (\Omega))$ and $\tilde{C} (T) = 0$.
\end{definition}

Note that since in our case the source term $S^\epsilon$ is given by (\ref{approxsourceterm}), the last term in (\ref{anisotropicWF_C}) is well defined. If we use a source function of a different and more general form, it suffices to be bounded in $L_t^2 L_x^2$ for the weak formulation to make sense.

We will omit the proof of the existence of weak solutions (see Definition \ref{definitionanisotropicWF}) for the anisotropic system (\ref{ANSC1}) - (\ref{ANSC_init}), (\ref{ANSC_bc}) since it follows straightforwardly by combining the standard finite-dimensional Galerkin approximation used for the existence of weak solutions of the Navier Stokes equations (see \cite{temam2001navier}) with the one used for the weak solutions for the parabolic equations. Now we are ready to state our main result.

\begin{theorem} \label{maintheorem}
Let $\mathbf{u}_0$ $\in$ $L^2 (\Omega)^3$, with $\nabla \cdot \mathbf{u}_0 = 0, \mathbf{u}_0 \cdot \mathbf{n} = 0$ on $\partial \Omega$; $\theta_H \in L^2(0,T; H^{-1/2}(\Gamma_G)),$ let $C_0 \in$ $L^2 (\Omega)$, $C_0 = 0$ on $ \partial \Omega$, and assume that the boundary conditions (\ref{ANSC_bc}) hold; then as the aspect ratio $\epsilon$ tends to zero, any weak solution $(\mathbf{u}^\epsilon, C^\epsilon)$ of the anisotropic equations (\ref{ANSC1}) - (\ref{ANSC5}) converge to a weak solution $(\mathbf{u}, C)$ of the hydrostatic equations of the polluted atmosphere (\ref{HNSC1}) - (\ref{HNSC5}).
\end{theorem}

\section{Proof of the main theorem} \label{proof}
 
This section is devoted to the proof of Theorem \ref{maintheorem}. In the extended case with the pollution equation the proof relies on a priori estimates gained from the energy inequality and, as we will see, this is sufficient to pass to the limit in the linear terms as $\epsilon$ vanishes. Concerning the nonlinear terms we will prove a uniform in $\epsilon$ space--time estimate for $\mathbf{u}^\epsilon$ and $C^\epsilon$ that will allow us to apply a compactness criterion proved in \cite{azerad2001mathematical}.

\subsection{Energy inequality}

As we mentioned before, one of the most fundamental tools in this proof is the energy inequality. If we derive it for the anisotropic system (\ref{ANSC1}) - (\ref{ANSC_init}), (\ref{ANSC_bc}), we obtain that for a.e. $t \in [0,T]$ the following holds:
\begin{equation}
\begin{split}  
& \frac{1}{2} (\norm{\mathbf{u}_H^\epsilon (t)}_{L^2}^2 + \epsilon ^ 2 \norm{u_3^\epsilon (t)}_{L^2}^2 + \norm{C^\epsilon (t)}_{L^2}^2) \\
& + \int_0 ^ t \big[ \norm{\nabla_\nu \mathbf{u}_H^\epsilon}_{L^2}^2 + \epsilon ^ 2 \norm{\nabla_\nu u_3^\epsilon}_{L^2}^2  + (\mathbf{M} \nabla C^\epsilon, \nabla C^\epsilon) \big] \diff \tau \\
& \leq \frac{1}{2} (\norm{\mathbf{u}_H^\epsilon (0)}_{L^2}^2 + \epsilon ^ 2 \norm{u_3^\epsilon (0)}_{L^2}^2 + \norm{C_0}_{L^2}^2) \\
& + \int_0^t \abs{\langle \theta_H, \mathbf{u}^\epsilon_H \rangle_{\Gamma_G}} \diff \tau + \int_0^t (S^\epsilon,C^\epsilon) \diff \tau. \raisetag{60pt} \label{energyIEQ}
\end{split}
\end{equation}

The right-hand side is bounded because of the hypothesis on the initial data, using that the $S^\epsilon$ source term is in $L^\infty_x (\Omega)$ and $\theta_H \in L^2(0,T; H^{-1/2}(\Gamma_G))$.

After applying the coercivity property (\ref{coercivity}) of the diffusion matrix, (\ref{energyIEQ}) now takes the form of the following energy inequality,
\begin{equation}
\begin{split}  
& \frac{1}{2} (\norm{\mathbf{u}_H^\epsilon (t)}_{L^2}^2 + \epsilon ^ 2 \norm{u_3^\epsilon (t)}_{L^2}^2 + \norm{C^\epsilon (t)}_{L^2}^2) \\
& + \int_0 ^ t \big [\norm{\nabla_\nu \mathbf{u}_H^\epsilon}_{L^2}^2 + \epsilon ^ 2 \norm{\nabla_\nu u_3^\epsilon}_{L^2}^2 + \lambda \norm{\nabla C^\epsilon}_{L^2}^2 \big ] \diff \tau \\
& \leq \frac{1}{2} (\norm{\mathbf{u}_H^\epsilon (0)}_{L^2}^2 + \epsilon ^ 2 \norm{u_3^\epsilon (0)}_{L^2}^2 + \norm{C_0}_{L^2}^2) \\
& + \int_0^t \abs{\langle \theta_H, \mathbf{u}^\epsilon_H \rangle_{\Gamma_G}} \diff \tau + \int_0^t (S^\epsilon,C^\epsilon) \diff \tau. \raisetag{60pt} \label{finalformEIEQ}
\end{split}
\end{equation}

\subsection{A priori estimates and weak convergence}

From the final form of the energy inequality we obtain uniform a priori estimates that we summarise in the following proposition.

\begin{prop}
Let $u^\epsilon_1, u^\epsilon_2, u^\epsilon_3$ and $C^\epsilon$ be the weak solutions of the system (\ref{ANSC1}) - (\ref{ANSC5}) in the sense of Definition \ref{definitionanisotropicWF}. Then it holds,

\begin{align}
& u^\epsilon_1, u^\epsilon_2, \epsilon u^\epsilon_3, C^\epsilon & \quad & \text{are bounded in $L^\infty (0,T; L^2(\Omega)),$} \label{proposition1} \\
& u^\epsilon_3 & \quad & \text{is bounded in $L^2 (0,T; L^2(\Omega)),$} \label{propositionExtra} \\
& u^\epsilon_1, u^\epsilon_2, \epsilon u^\epsilon_3 & \quad & \text{are bounded in $L^2(0,T; H^1(\Omega)),$} \label{proposition2} \\
& C^\epsilon & \quad & \text{is bounded in $L^2(0,T;H^1(\Omega)).$} \label{proposition3}
\end{align}

\end{prop}

\begin{proof}
By applying the classical Gr\"onwall inequality, (\ref{proposition1}), (\ref{proposition2}) and (\ref{proposition3}) follow directly by the energy inequality (\ref{finalformEIEQ}). On the other hand, (\ref{propositionExtra}) can be verified by using the divergence free condition and that as a consequence, we have $\partial_3 u_3 = - \partial_1 u_1 - \partial_2 u_2$ bounded in $L^2_t L^2 _x$. The bound holds for $u_3$ itself as well, owing to the Poincar\'e inequality in the vertical direction, and using that $u_3$ vanishes on the boundary.
\end{proof}

As a consequence, up to subsequences still denoted by the same way, we have the following weak convergence results.

\begin{align} 
& \mathbf{u}_H^\epsilon \rightharpoonup \mathbf{u}_H &\quad & \text{weakly in $L_t^\infty L_x^2 \cap L_t^2 H_x^1$,} \label{WeakConvUHEps} \\
& \epsilon^2 u_3^\epsilon \to 0 &\quad & \text{strongly in $L_t^\infty L_x^2 \cap L_t^2 H_x^1$,} \label{WeakConvU3Eps} \\ 
& u_3^\epsilon \rightharpoonup u_3 &\quad & \text{weakly in $L_t^2 L_x^2$,} \label{WeakConvU3NoEps} \\
& C^\epsilon \rightharpoonup C &\quad & \text{weakly in $L_t^\infty L_x^2 \cap L_t^2 H_x^1$} \label{WeakConvCEps1}
\end{align} 

\subsection{Passing to the limit}

In order to conclude the proof of Theorem \ref{maintheorem} we pass to the limit in the weak formulation (\ref{anisotropicWF_U})-(\ref{anisotropicWF_C}). To take the limit for the linear velocity terms that do not include the concentration function $C^\epsilon$ we can use (\ref{WeakConvUHEps}), (\ref{WeakConvU3Eps}) and (\ref{WeakConvU3NoEps}), and the bounds (\ref{propositionExtra}), (\ref{proposition2}). As for the nonlinear terms, the convergence can be shown using a time compactness theorem that is based on a result by Simon \cite{simon1986compact}. The key point is to apply a sort of generalisation of the classical translation criterium of Riesz-Frechet-Kolmogorov which enables us to get strong convergence for the horizontal velocities. This is achieved by establishing a bound for the perturbation of $\mathbf{u}_H$ of the form $\norm{\tau_h \mathbf{u}_H - \mathbf{u}_H}_{L^p(0,T-h, \mathbf{Y})} \leq \varphi(h) + \psi(\epsilon),$ where $\mathbf{Y}$ is a Banach space, $h$ is non-negative value, $\varphi$ and $\psi$ are appropriate functions with their limits vanishing at zero, and $\tau_h \mathbf{u}_H = \mathbf{u}_H (t+h) $. After obtaining strong convergence for $\mathbf{u}_H,$ the proof of the convergence for these nonlinear terms can be closed by applying basic interpolation techniques and the generalised Holder inequality. The details are omitted since they are analogous to those described in \cite{azerad2001mathematical}.

For the linear terms including the concentration $C^\epsilon$ we can directly apply the weak convergence result of (\ref{WeakConvCEps1}).

We get
\[ \int_0 ^ T (C^\epsilon , \partial_t \tilde{C}) d \mathrm{t} \to \int_0^T (C, \partial_t \tilde{C}) \diff t,\]
\[ \int_0 ^ T (\mathbf{M} \nabla C^\epsilon, \nabla \tilde{C}) d \mathrm{t} \to \int_0^T (\mathbf{M} \nabla C, \nabla \tilde{C}) \diff t.\]
For the source term, by applying standard results for distribution functions, we get $(S^\epsilon, \tilde{C}) \to (S, \tilde{C}).$ 

Finally we have to deal with the nonlinear term $(\mathbf{u}^\epsilon C^\epsilon, \nabla \tilde{C}).$

By analysing this nonlinear term in a more precise way we see that $u^\epsilon_i C^\epsilon$ for $i=1,2$ is weakly convergent since as we mentioned before, we have the strong convergence of the horizontal velocities, while in the case of $i=3$ this property does not hold. We overcome this difficulty by proving a compactness property for the pollution concentration $C^\epsilon.$ The key tool will be the following compactness criterion (for the proof see Theorem 5.1 in \cite{azerad2001mathematical}).

\begin{theorem}
Let $T>0,$ and let the Banach spaces $\mathbf{X} \hookrightarrow \mathbf{B} \hookrightarrow \mathbf{Y},$ where $\mathbf{X}$ is compactly embedded in $\mathbf{B},$ and $\mathbf{B}$ is continuously embedded in $\mathbf{Y}.$ Let $(f_\epsilon)_{\epsilon>0}$ be a family of functions of $L^p (0,T; \mathbf{X}), 1 \leq p \leq \infty,$ with the extra condition $(f_\epsilon)_{\epsilon>0} \subset C (0,T; \mathbf{Y})$ if $p = \infty,$ such that
\begin{enumerate}
\item $(f_\epsilon)_{\epsilon>0}$ is bounded in $L^p (0,T; \mathbf{X}),$
\item $\norm{\tau_h f_\epsilon - f_\epsilon}_{L^p(0,T-h; \mathbf{Y})} \leq \varphi(h) + \psi(\epsilon)$ for a given pair of functions $\varphi, \psi$ with
$$
\begin{cases}
\lim_{h \to 0} \varphi(h) = 0,\\
\lim_{\epsilon \to 0} \psi(\epsilon) = 0,
\end{cases}
$$
\end{enumerate}

where $\tau_h f_\epsilon$ denotes $f_\epsilon (t+h), h > 0.$

Then the family $(f_\epsilon)_{\epsilon>0}$ possesses a cluster point in $L^p (0,T; \mathbf{B})$ and also in $\mathcal{C}(0,T; \mathbf{B})$ if $p = \infty$ as $\epsilon \to 0.$
\end{theorem}

We will be able to apply the statement of this theorem for $p=2$ and the spaces $H^1 \hookrightarrow L^2 \hookrightarrow {H^2}^*,$ where ${H^2}^*$ is the dual of ${H^2}$, obtaining that
\[C^\epsilon \to C \text{ strongly in } L^2_t L^2_x,\]
and thus we have the weak convergence result $u^\epsilon_3 C^\epsilon \rightharpoonup u_3 C.$

In the following we will verify that the theorem's conditions hold for our choice of spaces and $p$. Since (\ref{proposition3}) is already given, we only have to show the bound $\norm{\tau_h C^\epsilon - C^\epsilon}_{L^2(0,T-h, {H^2}^*)} \leq \varphi(h) + \psi(\epsilon)$ in order to close the argument regarding the weak convergence of $u^\epsilon_3 C^\epsilon$.

\begin{prop}
The estimate $\norm{\tau_h C^\epsilon - C^\epsilon}_{L^2(0,T-h, {H^2}^*)} \leq c \cdot h^{1/4}$ holds.
\end{prop}

\begin{proof}
The spatial weak form of the anisotropic concentration equation (\ref{ANSC5}) is
\begin{equation}
\bigg (\frac{\partial C^\epsilon}{\partial t}, \tilde{C} \bigg ) - (\mathbf{u}^\epsilon C^\epsilon, \nabla \tilde{C}) + (\mathbf{M} \nabla C^\epsilon, \nabla \tilde{C}) = (S^\epsilon, \tilde{C}), \label{spatialWeakForm}
\end{equation}
where in this case the equality has to hold for all $\tilde{C} \in H^2_x.$

Now we take a test function $\tilde{C} \in H^2_x$ in (\ref{spatialWeakForm}) and integrate over $(t, t+h),$ i. e. we have
\[ (\tau_h C^\epsilon (t) - C^\epsilon (t), \tilde{C}) = \int_t^{t+h} \iota^\epsilon (s) \diff s,\]
where
\[ \iota^\epsilon (s) = (\mathbf{u}^\epsilon C^\epsilon, \nabla \tilde{C}) - (\mathbf{M} \nabla C^\epsilon, \nabla \tilde{C}) + (S^\epsilon, \tilde{C}). \]
In the next step we prove the bound
\begin{equation} \norm{\iota^\epsilon}_{L^{4/3}(0,T)} \leq c \big \| \tilde{C} \big \| _{H^2} \label{bound_iota} \end{equation}
for $\iota^\epsilon(s)$, where $c$ stands for a constant. In order to do this, we estimate each term of $\iota^\epsilon.$

\begin{align} 
& (S^\epsilon, \tilde{C}) &\quad & \text{is bounded in $L^2(0,T)$,} \\ 
& \label{lineartermboundgradC} (\mathbf{M} \nabla C^\epsilon, \nabla \tilde{C}) \leq c \big \| \tilde{C} \big \| _{H^1} \big \| C^\epsilon \big \| _{H^1} &\quad & \text{is bounded in $L^2(0,T)$,} \\
& (\mathbf{u}^\epsilon C^\epsilon, \nabla \tilde{C}) \leq \norm{\mathbf{u}^\epsilon}_{L^2} \norm{C^\epsilon}_{L^3} \big \| \tilde{C} \big \| _{H^2} &\quad & \text{is bounded in $L^{4/3}(0,T).$}
\end{align} 

The first bound regarding the source term holds because of the $L^2_t L^2_x$ regularity of the source (see (\ref{sourcedefinition}) and Remark \ref{sourceregularityremark}.); the bound of the linear term (\ref{lineartermboundgradC}) is a direct consequence of the Cauchy-Schwartz inequality; while the ultimate bound can be verified using (\ref{proposition1}) and (\ref{propositionExtra}), moreover considering the interpolation between $L^\infty(0,T; L^2)$ and $L^2(0,T; L^6),$ which yields that $C^\epsilon$ is in $L^4(0,T; L^3)$.

Finally, as we have now showed (\ref{bound_iota}), we apply the Holder inequality combined with this latter and arrive to
\[ \int_t^{t+h} \abs{\iota^\epsilon (s)} \diff s \leq c \big \| \tilde{C} \big \| _{H^2} h^{1/4},\]
which gives the proof of the lemma.
\end{proof}

Eventually, we obtain the strong convergence result $C^\epsilon \to C$ in $L^2(0,T; L^2),$ and the weak convergence $\mathbf{u}^\epsilon C^\epsilon \rightharpoonup \mathbf{u} C$ of the nonlinear term. This, combined with the previously described linear terms, concludes the proof of the main theorem.

\section{Concluding remarks} \label{concludingremarks}

We finish this paper briefly mentioning a few ideas regarding the physical aspects of the model and some additional choices concerning the source term specifically.

\begin{remark}
\normalfont
As previously explained in detail, the results obtained in this paper hold for a generic landscape. In order to achieve this, we use the $\nu_3 \partial_3 \mathbf{u}_H^\epsilon = \theta_H$ Neumann type boundary condition on $\Gamma_G.$ For the special case of inland domains we can switch the Neumann boundary condition to the $\mathbf{u}^\epsilon = 0$ Dirichlet type condition, paying attention to accordingly changing the definition of weak formulations as well.
\end{remark}

\begin{remark}
\normalfont In \cite{temam2005some} it is suggested that the domain we work on has the form of $p_0 < p < p_1$, because for the thin portion of air between the earth and the isobar $p = p_1$, another specific simplified model would be necessary. This means that a way to make the model more precise is to cut the domain at an isobar a little bit off the physical ground. This would of course radically change the boundary conditions we use for the functions (in fact it is not trivial how to change them for this case), because it would take away the advantage of having a wall-like, natural and physical boundary condition for the domain we work on. If we are positively above the ground, $\partial_3 C = 0$ would represent a strange reflective layer.
\end{remark}

\begin{remark}
\normalfont Note that the aspect ratio going to zero is not only legitimate on a global scale. When we switch from a global model to a local one in order to be able to use the beta plane approximation, we still have hundreds of thousands of square kilometres versus 1---5 kilometres.
\end{remark}

\begin{remark}
\normalfont We chose the Gaussian approximated delta function to give shape to the source term, but it can be defined in several different ways using the many different pulses that approximate the delta function (choosing only those however that are physically meaningful in our case, for example its values are nonnegative), e.g.

\begin{enumerate}
\item the unit impulse: $\delta_\epsilon (\mathbf{x}-\mathbf{x_s}) = \epsilon/2, \abs{(\mathbf{x}-\mathbf{x_s})} < 1/\epsilon $,
\item the Lorentzian pulse: $\delta_\epsilon (\mathbf{x}-\mathbf{x_s}) = \epsilon / (1 + \pi^2 \epsilon^2 (\mathbf{x}-\mathbf{x_s})^2$).
\end{enumerate}

The Gaussian and Lorentzian version of the source term is continuous and bounded on $\Omega$, it is actually in $L^\infty$ in space (and in $L^2$ as a consequence), while it is easy to see that the piecewise constant unit pulse is in $L^2$ as well. Since the time dependence of this source is described by the Heaviside function, the time norm is also finite.

\end{remark}

\begin{remark}
\normalfont Although the results of this article were achieved in a classical, fixed Cartesian coordinate system, there are other possibilities to build up the model in. We will briefly consider two of these.

\begin{enumerate}
\item One option is to fix the coordinate system in such a way that the $x$ axis matches the downwind direction (\cite{goyal2011mathematical}). In other words, $z$ is the vertical direction, while $x$ is the main wind, and $y$ is the crosswind direction. It is a legitimate assumption if we are considering a time interval which is not too long in the sense that it is reasonable to assume that we have a permanent, relatively stable downwind direction. For this scenario we consider a diagonal diffusion matrix, i.e. diffusion coefficients $K_x, K_y, K_z,$ and more strict boundary conditions, namely we require $\partial_2 C$ and $\partial_3 C$ to vanish on the lower boundary. This is necessary because of the changes in the energy inequality caused by the additional $\epsilon$ term that we will introduce in the following. The adjusted coordinate system makes it possible to add another equation in the scaling process, which is different in nature from the previous ones in the sense that it is not derived from the shallowness of the domain but rather suggested by the fact that in the downwind direction it is natural to assume that the diffusion is negligible compared to downwind transport, i.e.
\begin{equation} \abs{u_1 \partial_x C} \gg \abs{ \textit{K}_{x} \partial_{xx} ^2 C }\end{equation}
which leads us to using
\[ K_x = \epsilon M_1 \]
in this convection dominated scenario.

Adding this equation on the one hand makes the concentration equation in the hydrostatic limit model one term shorter, since the first diffusion term drops out. On the other hand it leads to a new mathematical challenge handling the weak convergence of the nonlinear term $\mathbf{u}^\epsilon C^\epsilon$ in a situation where we do not have the $H^1_x$ regularity of $C^\epsilon$ --- as we have $L^2_x$ boundedness only for $\sqrt{\epsilon} \partial_1 C^\epsilon$, and not for $\partial_1 C^\epsilon$ itself. We plan to investigate this scenario in an upcoming article.

\item Another possibility is to adjust the coordinate system in a way that the $x$ axis matches not with the downwind, but the wind velocity vector function $\mathbf{u}^\epsilon$ itself. This option brings along a $\mathbf{u}^\epsilon$-dependent coordinate transform matrix in every equation, however it also makes the velocity vector one dimensional in the new coordinate system, which can be a potential benefit in a computational application.

\end{enumerate}

\end{remark}

\begin{remark}
\normalfont Note that the steps leading to the final form of the energy inequality remain true even if we use a diffusion matrix function with non-constant coefficients of the form $M(x)$ and such that the uniform ellipticity condition (\ref{coercivity}) holds in the form

\begin{equation}
   \sum_{i,j = 1}^3 M_{ij}(x) \xi_i \xi_j \geq \lambda \norm{\xi}^2
\end{equation}

for almost any $x \in \Omega$ and all $\xi \in \mathbb{R}^3$, $\lambda \in \mathbb{R}, \lambda > 0.$
\end{remark}

\appendix

\section{The existence of weak solution of the hydrostatic system of the polluted atmosphere}

We finish the paper by pointing out that the existence of weak solutions for the hydrostatic system of the polluted atmosphere can be proved directly and not only as the result of the small aspect ratio limit of the anisotropic Navier-Stokes system. In this section, for completeness, we sketch this proof of existence.

The idea is that the polluted atmosphere can be viewed as the air-analogous version of the salty sea water in terms of the equations that represent it. In both cases we have a fluid which can be seen approximately as a homogenous incompressible fluid carrying some other particles in a significant concentration, the basis fluid motion is well described by the incompressible Navier-Stokes equations, and the additional concentration is described by an advection-diffusion equation. 

The paper of \cite{temam2005some} provides an existence result for an arbitrary time interval $[0, t_1]$ for the system describing the salty ocean water with temperature. The structure of the equations themselves describing the system is essentially the same, if we drop the equation describing the temperature, we basically arrive to the same system that describes the polluted atmosphere.

The main difference consists in the boundary conditions: on every boundary they use a zero flux for the salinity. This coincides with our lower boundary conditions, but on the lateral boundary we use a fix constant.

We will follow step by step the structure of the proof in \cite{temam2005some}, redefining the necessary quantities to be valid for our case and verify that the proof holds in the case of the polluted atmosphere as well (this is necessary not only because of using altered boundary conditions but also because the regularities of the functions are different in some cases).

\textbf{Step 1.} Firstly we redefine the norms and the functional setting in order to make them match our system.

Throughout this section we will use the $V = V_1 \times V_2, H = H_1 \times H_2, \mathcal{V} = \mathcal{V}_1 \times \mathcal{V}_2,$ and $V_{(2)}$ function spaces, where $V$ and $H$ essentially represent $H^1$ and $L^2$ regularities, while $\mathcal{V}$ and $V_{(2)}$ stand for a subset of the $C^{\infty}$ function space and its closure in $H^2,$ respectively. In more detail, we introduce

$$ V_1 = \{ \mathbf{u}_H \in H^1_{\Gamma_A} (\Omega) \times H^1_{\Gamma_A} (\Omega), \nabla \cdot \int_h^0 \mathbf{u}_H \diff z = 0 \}, $$

$$ V_2 = H^1(\Omega), $$

$$ H_1 = \{ \mathbf{u}_H \in H^0_{\Gamma_A} (\Omega) \times H^0_{\Gamma_A} (\Omega), \nabla \cdot \int_h^0 \mathbf{u}_H \diff z = 0 \}, $$
and
$$ H_2 = L^2(\Omega) $$

Let $\mathcal{V}_1$ be the space of $C^\infty$ vector functions $\mathbf{u}_H$ which vanish in a neighborhood of $\Gamma$ and such that $$\nabla \cdot \int_h^0 \mathbf{u}_H \diff z = 0,$$ and let $\mathcal{V}_2$ denote the $C^\infty$ functions on $\bar{\Omega}$.

Finally we introduce the space $V_{(2)}:$
$$ V_{(2)} \text{ is the closure of } \mathcal{V} \text{ in } (H^2(\Omega))^3.$$

We denote by $U$ the vector $(\mathbf{u}_H, C).$ The scalar products and norms in the case of the polluted atmosphere are defined as follows.

\[ ((U, \tilde{U})) = ((\mathbf{u}_H, \mathbf{\tilde{u}}_H))_1 + K_C ((C, \tilde{C}))_2 \]

\[ ((\mathbf{u}_H, \mathbf{\tilde{u}}_H))_1 = \int_\Omega \nabla_\nu \mathbf{u}_H \nabla_\nu \mathbf{\tilde{u}}_H \diff \omega \]

\[ ((C, \tilde{C}))_2 = \int _\Omega \big (\nabla C)^T \mathbf{M} \nabla \tilde{C} \diff \omega \]

\[\norm{U} = ((U,U)) ^ {1/2}\]

\[ (U, \tilde{U})_H = \int _\Omega \big [ -\mathbf{u}_H \mathbf{\tilde{u}}_H - K_C C \tilde{C} \big ] \diff \omega \]

\[\abs{U}_H = \abs{(U,U)_H}^{1/2}\]

\textbf{Step 2.} 
Considering a sufficiently regular test function $\tilde{U}=(\mathbf{\tilde{u}}_H, \tilde{C})\in V$ now the hydrostatic weak formulation of the polluted atmosphere can be written as
\begin{equation} \label{EXISTENCEFORMAT}
\bigg ( \frac{\diff}{\diff t} U, \tilde{U} \bigg ) _H + a(U, \tilde{U}) + b(U,U,\tilde{U}) + e(U, \tilde{U}) = l(\tilde{U}) \\
\end{equation}
with
\[ a(U,\tilde{U}) = a_1 (U, \tilde{U}) + K_C a_2 (U, \tilde{U})\]

\[ a_1 (U, \tilde{U}) = \int_\Omega \nabla_\nu \mathbf{u}_H \nabla_\nu \mathbf{\tilde{u}}_H \diff \omega \]

\[ a_2 (U, \tilde{U}) = \int_\Omega (\nabla C)^T \mathbf{M} \nabla \tilde{C} \diff \omega \]

\[ b = b_1 + K_C b_2\]

\[ b_1 (U, \tilde{U}, U^\#) = - \int_\Omega (\mathbf{u}_H \mathbf{\tilde{u}}_H + \mathbf{u}_H \tilde{u}_3) \nabla \mathbf{u}_H^\# \]

\[ b_2 (U, \tilde{U}, U^\#) = - \int_\Omega (\mathbf{u}_H \tilde{C} + u_3 \tilde{C}) \nabla C^\# \]

\[ e(U, \tilde{U}) = \int_\Omega b(\mathbf{u}_H) \mathbf{\tilde{u}}_H \diff \omega \]

\[ l(\tilde{U}) = - \int_\Omega S \tilde{C} \diff \omega + \langle \theta_H, \tilde{u}_H \rangle _{\Gamma_G}. \]

Here the term $b(\mathbf{u}_H)$ represents the $\alpha (-u_2, u_1)$ vector, and we use the fact that the vertical velocity $w$ can be expressed as $u_3 = u_3(\mathbf{u}_H)$ because the divergence of the three dimensional velocity is zero.

\textbf{Step 3.} It is straightforward to see that $a$ is trilinear continuous, the coercivity of $a$ is also trivial, moreover $e$ is bilinear continuous and $e(U,U) = 0$.

What changes slightly is the verification of the results available on the form $b,$ since the space derivatives in our case are applied on the test function. In the first two points below we discuss the difference in the proof that guarantees the trilinear continuity of $b,$ that is,
\begin{equation} \label{trilincontb}
\lvert b(U, \tilde{U}, U^\# ) \rvert \leq \begin{cases}
c \norm{U} \big \|\tilde{U}\big \| \big \| U^\# \big \|_{V_{(2)}}, & \forall \text{ } U, \tilde{U} \in V, U^\# \in V_{(2)}, \\
c \norm{U} \big \|\tilde{U}\big \| _{V_{(2)}} \big \| U^\# \big \|, & \forall \text{ } U, U^\# \in V, \tilde{U} \in V_{(2)};
\end{cases}
\end{equation}
finally we also adjust the original steps that provide the antisymmetric property $b(U, \tilde{U}, U^\# ) = - b(U, U^\#, \tilde{U} )$ and an improvement of (\ref{trilincontb}) in this new environment.

\begin{itemize}
\item To prove the first bound for the case of $U, \tilde{U} \in V, U^\# \in V_{(2)}$ in (\ref{trilincontb}), let us consider the typical term $$ \int_\Omega u_3 \tilde{C} \nabla C^\# d \omega.$$ We observe that we need to change the original choice of constants $(p_1 = p_2 = 2, p_3 = \infty)$ for applying the Holder inequality. We can not use the $L_\infty$ bound on the third term, since we would eventually arrive to $\big \| C^\# \big \| _{H^3}$, which would prevent us from closing the bound. Instead, we can use

\[ \int_\Omega u_3 \tilde{C} \nabla C^\# d \omega \leq c \norm{u_3}_{L^2} \big \| \tilde{C} \big \| _{L^4} \big \| \nabla C^\# \big \| _{L^4} \leq \]
\[ \leq c \norm{U} \big \|\tilde{U}\big \| \big \|\nabla C^\# \big \|_{H^1} \leq c \norm{U} \big \|\tilde{U}\big \| \big \| U^\# \big \|_{V_{(2)}}, \]

which yields the needed result.

\item Similarly, for the case of $U, U^\# \in V, \tilde{U} \in V_{(2)}$ in the second inequality of (\ref{trilincontb}), we need to adjust the space choices according to the present scenario. With the updated $L^p$ constants we arrive to the

\[ \int_\Omega u_3 \tilde{C} \nabla C^\# d \omega \leq c \big \| \tilde{C} \big \| _{L^\infty} \norm{ u_3}_{L^2}\big \| \nabla C^\# \big \| _{L^2} \leq c \norm{U} \big \|\tilde{U}\big \| _{V^{(2)}} \big \| U^\# \big \| \]

new inequality that provides the continuity property.

\item The proof of $$b(U, \tilde{U}, U^\# ) = - b(U, U^\#, \tilde{U} ) \text{ for } U, \tilde{U}, U^\# \in V \text{ and } \tilde{U} \text{ or } U^\# \text{ in } V_{(2)}$$ can be derived by integration by parts as we have

\[ \int_0 ^T (\mathbf{u} \nabla C, \tilde{C}) \diff t = - \int_0 ^T ( \mathbf{u} C, \nabla \tilde{C}) \diff t, \]

observing that the boundary terms disappear.

\item To establish the improvement $$ \lvert b(U, \tilde{U}, U^\#) \rvert \leq c \norm{U} \lvert \tilde{U} \rvert _H^{1/2} \big \| \tilde{U} \big \| ^{1/2} \big \| U^\# \big \| _{V_{(2)}} $$ of the first bound for $U, \tilde{U} \in V, U^{\#} \in V_{(2)}$, we consider the most typical term and we bound it by

\[ \int_\Omega u_3 \tilde{C} \nabla C^\# \diff \omega \leq \norm{u_3}_{L^2} \big \| \tilde{C} \big \| _{L^3} \big \| \nabla C^\# \big \| _{L^6}.\]

Note that since our definition of the form $b$ is different from \cite{temam2005some}, thus we need to skip the switch of function arguments caused by the application of $ \lvert b(U, \tilde{U}, U^\#) \rvert = \lvert b(U, U^\#, \tilde{U}) \rvert$, otherwise the bound can not be closed.

Now, using the $\norm{\varphi}_{L^3} \leq c \norm{\varphi}_{L^2}^{1/2} \norm{\varphi}_{H^1}^{1/2}$ Sobolev-Ladyzhenskaya inequality in space dimension $3,$ we can bound this term by

\[ c \norm{\mathbf{u}_H} \big \| U^\# \big \| _{V_{(2)}} \lvert \tilde{U} \rvert ^ {1/2} \big \| \tilde{U} \big \| ^{1/2} ,\]

which gives the inequality.

\end{itemize}

This means that all the necessary properties of the forms $a, b$ and $e$ are equally valid for our set of equations (\ref{HNSC1}) - (\ref{HNSC_bc}), and we can use them to prove a stationary-case existence theorem identically as the proof is constructed in \cite{temam2005some}. 

The weak formulation is as follows.



\begin{definition}
Given a source term $S = \delta$ or $S \in L^2(\Omega),$ and a wind traction parameter $\theta_H \in L^2(0,T;H^{-1/2}(\Gamma_G)),$ find $U = (\mathbf{u}_H, C),$ that satisfies the hypothesis of Definition \ref{definitionhydrostaticWF} in the space dimensions, moreover
\begin{equation} \label{FWSTAT} a(U, \tilde{U}) + b(U,U, \tilde{U}) + e(U, \tilde{U}) = l(\tilde{U}) \end{equation}
for every $\tilde{U} \in V_{(2)}.$
\end{definition}

Note that the definition of the weak formulation in our case is a bit more general, since we allow not only $L^2$ functions to represent the source term, but the $\delta$ distribution as well. As we have seen before, the weak formulation is still valid in this case because of the basic properties of $\delta,$ so this generalisation is mathematically valid. 

We have the following theorem establishing the existence of solutions of the stationary equations, which can be proved by Galerkin method using the fundamental results we verified for the functions $a,b$ and $e$.

\begin{theorem}
We are given $S = \delta$ or $S$ in $L^2(\Omega),$ and a wind traction parameter $\theta_H \in L^2(0,T;H^{-1/2}(\Gamma_G)),$ then the problem (\ref{FWSTAT}) possesses at least one solution $U \in V$ such that

\[ \norm{U} \leq \frac{1}{c_1} \norm{l}_{V'},\] where $c_1$ is the coercivity constant.

\end{theorem}

\textbf{Step 4.} In this step we provide the operator form of the equation (\ref{EXISTENCEFORMAT}) in the Hilbert space $V'_{(2)}$.

The operators $ \frac{\diff U}{\diff t}, A, B$ and $E$ can be set analogously as in the original version; $A$ for example is defined to be a linear continuous form from $V$ into $V',$ given by the formula

\[ \langle AU, \tilde{U} \rangle = a(U, \tilde{U}), \quad \forall U, \tilde{U} \in V.\]

The operator form of equation (\ref{EXISTENCEFORMAT}) becomes

\[ \frac{\diff U}{\diff t} + AU + B(U,U) + EU = l.\]

\textbf{Step 5.} Finally we discuss the details of the time-dependent case, where we establish the existence, for all time, of the solutions for the system (\ref{HNSC1})-(\ref{HNSC5}).

The weak formulation in the time-dependent case takes the following form.

\begin{definition}

Given $T,$ $U_0 \in H,$ $\theta_H \in L^2(0,T;H^{-1/2}(\Gamma_G)),$ and a source term $S = \delta$ or $S \in L^2(\Omega),$ find $U$ that satisfies the hypothesis of Definition \ref{definitionhydrostaticWF}, moreover
\begin{equation} \label{timedependentWFE}
\begin{split}
\bigg ( \frac{\diff}{\diff t} U, \tilde{U} \bigg ) + a(U, \tilde{U}) + b(U,U, \tilde{U})+ e(U, \tilde{U}) & = l(\tilde{U}) \quad \forall \tilde{U} \in V_{(2)} \\
U(0) &= U_0
\end{split}
\end{equation}
\end{definition}

We proceed by using finite differences in time, taking an arbitrary integer $N,$ setting a time step $k = \Delta t = T/N,$ and recursively for $n = 1, ... , N$ considering

\[ U^0 = U_0,\]
\begin{equation}
\begin{split}
\frac{1}{\Delta t} (U^n - U^{n-1}, \tilde{U})_H & + a(U^n, \tilde{U}) + b(U^n,U^n, \tilde{U})\\
& + e(U^n, \tilde{U}) = l^n(\tilde{U}), \quad \text{ for any } \tilde{U} \in V_{(2)} \\
\end{split}
\end{equation}

Following the proof of the existence of $U^n \in V,$ proving some a priori estimates, introducing approximate functions on (0,T), and using the Aubin compactness theorem in an identical way as described in \cite{temam2005some}, we pass to the limit and arrive to

\[ - \int_0^T (U, \tilde{U})_H \psi ' \diff t + \int_0^T \big [ a(U, \tilde{U}) + b(U,U, \tilde{U}) + e(U, \tilde{U}) \big ] \psi \diff t = \]
\[ = (U_0, \tilde{U}) \psi (0) + \int_0 ^ T l(\tilde{U}) \psi \diff t, \]
where $\psi$ is a function defined on $[0,T]$, used to create the time-dependent version $\tilde{U}\psi$ of the test function. Renaming this product simply to be $\tilde{U}(t)$, the time-dependent test function, we arrive exactly to

\[ - \int_0^T (U, \partial_t \tilde{U})_H \diff t + \int_0^T \big [ a(U, \tilde{U}) + b(U,U, \tilde{U}) + e(U, \tilde{U}) \big ] \diff t = \]
\[ = (U_0, \tilde{U}(0)) + \int_0 ^ T l(\tilde{U}) \diff t, \]
which is the same as (\ref{timedependentWFE}), and which thus finally leads us to the main existence result of this section.

\begin{theorem}

Given $T > 0$, $U_0 \in H,$ $\theta_H \in L^2(0,T; H^{-1/2}(\Gamma_G))$, and a source term $S = \delta$ or $S \in L^2(\Omega),$ then there exists $(\mathbf{u}, C)$ with $u_3 = - \int_0^T \nabla \cdot \mathbf{u}_H$, $\mathbf{u} = (\mathbf{u}_H, u_3) \in L^2 (0, T; \mathbf{W})$ with $\mathbf{u}_H \in L^\infty (0,T; L^2(\Omega)^2),$ and $C \in L^\infty (0,T; L^2 (\Omega)) \cap L^2( 0,T; H^1(\Omega))$, which is solution to (\ref{hydrostaticWF_U})-(\ref{hydrostaticWF_C}).

\end{theorem}

\nocite{*}
\bibliographystyle{abbrv}

\begin{thebibliography}{10}

\bibitem{adams2003sobolev}
R.~A. Adams and J.~J. Fournier.
\newblock {\em Sobolev spaces}, volume 140.
\newblock Academic press, 2003.

\bibitem{ali1998partial}
I.~Ali, S.~Kalla, and H.~Khajah.
\newblock A partial differential equation related to a problem in atmospheric pollution.
\newblock {\em Mathematical and computer modelling}, 28(12):1--6, 1998.

\bibitem{anderson1995computational}
J.~D. Anderson and J.~Wendt.
\newblock {\em Computational fluid dynamics}, volume 206.
\newblock Springer, 1995.

\bibitem{andria2000mathematical}
G.~Andria, A.~Lay-Ekuakille, and M.~Notarnicola.
\newblock Mathematical models for atmospheric and industrial pollutant prediction.
\newblock In {\em XVI IMEKO World Congress, Vienna, Austria}, 2000.

\bibitem{arya2001introduction}
P.~S. Arya.
\newblock {\em Introduction to micrometeorology}, volume~79.
\newblock Elsevier, 2001.

\bibitem{azerad2001mathematical}
P.~Az{\'e}rad and F.~Guill{\'e}n.
\newblock Mathematical justification of the hydrostatic approximation in the primitive equations of geophysical fluid dynamics.
\newblock {\em SIAM Journal on Mathematical Analysis}, 33(4):847--859, 2001.

\bibitem{besson1992some}
O.~Besson and M.~Laydi.
\newblock Some estimates for the anisotropic navier-stokes equations and for the hydrostatic approximation.
\newblock {\em ESAIM: Mathematical Modelling and Numerical Analysis}, 26(7):855--865, 1992.

\bibitem{swart2016utrecht}
H.~E. de~Swart.
\newblock The governing equations and the dominant balances of flow in the atmosphere and ocean.
\newblock {\em USPC Summerschool, Utrecht}, 2016.

\bibitem{goyal2011mathematical}
P.~Goyal and A.~Kumar.
\newblock Mathematical modeling of air pollutants: An application to Indian urban city.
\newblock In {\em Air Quality-Models and Applications}. InTech, 2011.

\bibitem{hosseini2013}
B.~Hosseini.
\newblock Dispersion of pollutants in the atmosphere: a numerical study.
\newblock Master's thesis, Simom Fraser University, 2013.

\bibitem{kathirgamanathan2002source}
P.~Kathirgamanathan, R.~McKibbin, and R.~McLachlan.
\newblock Source term estimation of pollution from an instantaneous point source.
\newblock {\em Research Letters in Information Mathematic Science, Vol. 3, pp. 59-67.}, 2002.

\bibitem{lions1996mathematical}
P.-L. Lions.
\newblock {\em Mathematical Topics in Fluid Mechanics: Volume 2: Compressible Models}, volume~2.
\newblock Oxford University Press on Demand, 1996.

\bibitem{marsaleix2006considerations}
P.~Marsaleix, F.~Auclair, and C.~Estournel.
\newblock Considerations on open boundary conditions for regional and coastal ocean models.
\newblock {\em Journal of Atmospheric and Oceanic Technology}, 23(11):1604--1613, 2006.

\bibitem{monin1959boundary}
A.~Monin.
\newblock On the boundary condition on the earth surface for diffusing pollution.
\newblock In {\em Advances in Geophysics}, volume~6, pages 435--436. Elsevier, 1959.

\bibitem{pedlosky2013geophysical}
J.~Pedlosky.
\newblock {\em Geophysical fluid dynamics}.
\newblock Springer Science \& Business Media, 2013.

\bibitem{prodanova2008application}
M.~Prodanova, J.~L. Perez, D.~Syrakov, R.~San~Jose, K.~Ganev, N.~Miloshev, and S.~Roglev.
\newblock Application of mathematical models to simulate an extreme air pollution episode in the Bulgarian city of Stara Zagora.
\newblock {\em Applied Mathematical Modelling}, 32(8):1607--1619, 2008.

\bibitem{simon1986compact}
J.~Simon.
\newblock Compact sets in the space $L^p(0, t; B)$.
\newblock {\em Annali di Matematica pura ed applicata}, 146(1):65--96, 1986.

\bibitem{temam2001navier}
R.~Temam.
\newblock {\em Navier-Stokes equations: Theory and numerical analysis}, volume 343.
\newblock American Mathematical Soc., 2001.

\bibitem{temam2005some}
R.~Temam and M.~Ziane.
\newblock Some mathematical problems in geophysical fluid dynamics.
\newblock {\em Handbook of mathematical fluid dynamics}, 3:535--658, 2005.

\bibitem{wang2014ocean}
Q.~Wang, W.~Zhou, D.~Wang, and D.~Dong.
\newblock Ocean model open boundary conditions with volume, heat and salinity conservation constraints.
\newblock {\em Advances in Atmospheric Sciences}, 31(1):188--196, 2014.

\bibitem{yeh1975three}
G.-T. Yeh and C.-H. Huang.
\newblock Three-dimensional air pollutant modeling in the lower atmosphere.
\newblock {\em Boundary-Layer Meteorology}, 9(4):381--390, 1975.

\bibitem{zhuk2016source}
S.~Zhuk, T.~T. Tchrakian, S.~Moore, R.~Ord{\'o}n?ez-Hurtado, and R.~Shorten.
\newblock On source-term parameter estimation for linear advection-diffusion equations with uncertain coefficients.
\newblock {\em SIAM Journal on Scientific Computing}, 38(4):A2334--A2356, 2016.

\end{thebibliography}

\end{document}